\documentclass[12pt]{amsart}

\usepackage{amsfonts}
\usepackage{amsthm, amssymb, amsmath}
\usepackage{epstopdf}
\usepackage{fullpage}
\usepackage{mathrsfs}
\usepackage{enumerate}
\usepackage{enumitem}
\usepackage{mathrsfs}
\usepackage{listings}
\usepackage{mathtools}
\usepackage[all]{xy}
\usepackage[utf8]{inputenc}
\usepackage[english]{babel}
\usepackage{tikz-cd}
\usepackage{hyperref}
\usepackage{color}
\usepackage[normalem]{ulem}
\makeatletter
\newtheorem*{rep@theorem}{\rep@title}
\newcommand{\newreptheorem}[2]{%
\newenvironment{rep#1}[1]{%
 \def\rep@title{#2 \ref{##1}}%
 \begin{rep@theorem}}%
 {\end{rep@theorem}}}
 
\newtheorem*{rep@corollary}{\rep@title}
\newcommand{\newrepcorollary}[2]{%
\newenvironment{rep#1}[1]{%
 \def\rep@title{#2 \ref{##1}}%
 \begin{rep@corollary}}%
 {\end{rep@corollary}}}

\newtheorem*{rep@conjecture}{\rep@title}
\newcommand{\newrepconjecture}[2]{%
\newenvironment{rep#1}[1]{%
 \def\rep@title{#2 \ref{##1}}%
 \begin{rep@conjecture}}%
 {\end{rep@conjecture}}}
\makeatother

\newtheorem{thm}{Theorem}[section]
\newreptheorem{theorem}{Theorem}
\newtheorem{prop}[thm]{Proposition}
\newtheorem{lem}[thm]{Lemma}
\newtheorem{cor}[thm]{Corollary}
\newrepconjecture{corollary}{Corollary}
\newtheorem{conj}[thm]{Conjecture}
\newrepconjecture{conjecture}{Conjecture}
\newtheorem{question}[thm]{Question}
\newtheorem*{principle}{Local-Global Principle}

\newtheorem{definition}[thm]{Definition}
\newtheorem{example}[thm]{Example}
\newtheorem*{remark}{Remark}

\newcommand{\QQ}{\mathbb{Q}}
\newcommand{\ZZ}{\mathbb{Z}}
\newcommand{\FF}{\mathbb{F}}

\newcommand{\p}{\mathfrak{p}}
\newcommand{\B}{\mathfrak{P}}
\newcommand{\Gal}{\operatorname{Gal}}
\newcommand{\E}{\mathcal{E}}

\newcommand{\D}{\mathcal{D}}
\newcommand{\I}{\mathcal{I}}
\newcommand{\J}{\mathcal{J}}
\newcommand{\G}{\mathcal{G}}
\newcommand{\OO}{\mathcal{O}}
\newcommand{\N}{\operatorname{N}}
\newcommand{\GL}{\operatorname{GL}}
\newcommand{\SL}{\operatorname{SL}}
\newcommand{\Mod}{\operatorname{mod }}
\newcommand{\tr}{\operatorname{tr}}
\newcommand{\Frob}{\operatorname{Frob}}
\newcommand{\Aut}{\operatorname{Aut}}
\newcommand{\Hom}{\operatorname{Hom}}

\makeatletter
\let\c@equation\c@thm
\makeatother
\numberwithin{equation}{section}

\title{The Realizability Problem with Inertia Conditions}
\author{Yuan Liu}

\begin{document}

\maketitle
\makeatletter{\renewcommand*{\@makefnmark}{}
\makeatother}

\begin{abstract}
In this paper, we consider the inverse Galois problem with described inertia behavior. For a finite group $G$, one of its subgroups $I$ and a prime integer $p$, we ask whether or not $G$ and $I$ can be realized as the Galois group and the inertia subgroup at $p$ of an extension of $\QQ$. We first discuss the result when $G$ is an abelian group. Then in the case that $G$ is of odd order, Neukirch showed that there exists such an extension if and only if the given inertia condition is realizable over $\QQ_p$, from which we obtain the answer for this case by studying the structure of extensions of $\QQ_p$ and applying techniques from embedding problems. As a corollary, we give an explicit presentation of the Galois group of the maximal pro-odd extension of $\QQ_p$. When $G=\GL_2(\FF_p)$ for an odd prime $p$, we relate our realizability problem to modular Galois representations and use elliptic curves to give answers for those subgroups $I$ corresponding to weight 2 modular forms. Finally, we provide an example arising from Grunwald-Wang's counterexample for which the local-global principle of our realizability problem fails.
\end{abstract}

\section{Introduction}\label{s1}
In this paper, we consider a refined version of the inverse Galois problem stated as follows.
\begin{question} \label{question}
 Given a finite group $G$, a subgroup $I$ and a prime integer $p$, does there exist a $G$-extension of $\QQ$ whose inertia subgroup at $p$ is $I$?
\end{question}
If there exists an extension $K/\QQ$ satisfying the conditions in Question \ref{question}, then the triple $(G,I,p)$ is called \emph{$\QQ$-realizable}, and we say that the field extension $K/\QQ$ \emph{realizes} $(G,I,p)$. For such $K/\QQ$, its local completion at $p$ is an extension $K_p/\QQ_p$ with inertia subgroup $I$ and Galois group $D$, where $D$ is the decomposition subgroup of $K/\QQ$ at $p$. Thus, the $\QQ$-realizability of $(G,I,p)$ implies the $\QQ_p$-realizability of $(D,I)$ for some subgroup $D$ of $G$ containing $I$ (where \emph{$\QQ_p$-realizability} is defined similarly, see Definition \ref{realizability}). This leads us to ask whether the converse, the local-global principle described below, is valid.

\begin{principle}
The triple $(G,I,p)$ is $\QQ$-realizable if there exists a subgroup $D$ of $G$ such that $D$ contains $I$ and $(D,I)$ is $\QQ_p$-realizable.
\end{principle}

In Section \ref{2}, we begin with the definition of decomposition and inertia subgroups, and go over some results about  the structure of local field extensions and embedding problems which will be used in the later sections. In Section \ref{3}, we answer Question \ref{question} for the case when the given $G$ is abelian by proving the following theorem.

\begin{thm}\label{thm_abelian}
  When $G$ is finite abelian, $(G,I,p)$ is $\QQ$-realizable if and only if $I$ is a quotient of $\ZZ_p^\times$.
\end{thm}

By local class field theory, the inertia subgroup of the maximal abelian extension $\QQ_p^{ab}$ of $\QQ_p$ is isomorphic to $\ZZ_p^{\times}$,
whence Theorem \ref{thm_abelian} shows that the local-global principle is valid when $G$ is finite abelian.

In Section \ref{4}, we discuss the $\QQ$-realizability problem when the order of $G$ is odd. In this case, Neukirch \cite{Solvable} showed that for a given local extension $K_p/\QQ_p$, there exists a $G$-extension $K/\QQ$ whose local completion is $K_p/\QQ_p$ if and only if $\Gal(K_p/\QQ_p)$ can be embedded into $G$ (see Lemma \ref{Neukirch}). It follows that the local-global principle is also valid in this case and it remains to study the realizability problem over $\QQ_p$. By applying techniques from embedding problems, we obtain the following critertion.

\begin{thm} \label{conclusion_odd}
  Assume $G$ is a group of odd order. $(G,I,p)$ is $\QQ$-realizable if and only if there is a subgroup $D$ of $G$ containing $I$ such that 
  \begin{enumerate}
    \item (Tame Condition) $I$ is a normal subgroup of $D$ and has a normal Sylow $p$-subgroup, denoted by $I_p$. Moreover, there exist $\sigma, \tau \in D/I_p$ such that
$$D/I_p = \langle \sigma, \tau \mid \tau^e=1, \sigma^f=\tau^r, \tau^\sigma=\tau^p \rangle,$$
and $I/I_p$ is the subgroup generated by $\tau$.
    \item (Wild Condition) $I_p$ is generated by one element and its conjugates by $D$, i.e. 
    $$I_p = \langle a \rangle^{D}.$$
  \end{enumerate}
\end{thm}
Furthermore, we show that Theorem \ref{conclusion_odd} can be used to explicitly determine the Galois group of the maximal pro-odd extension of $\QQ_p$. 

\begin{cor} \label{pro-odd}
The Galois group of the maximal pro-odd extension of $\QQ_p$ is the pro-odd group topologically generated by three elements $\sigma, \tau, x$ with the following defining relations.
\begin{enumerate}
\item The wild inertia subgroup is the closed normal subgroup generated by $x$, which is a free pro-$p$ group.
\item The elements $\sigma, \tau$ satisfy the tame relation
$$\tau^\sigma=\tau^p.$$
\end{enumerate}
\end{cor}
When $p$ is odd, Corollary \ref{pro-odd} can also be proven using the presentation of $G_{\QQ_p}$ given by Jannsen and Wingberg \cite{J-W, Neftin}.

In Section \ref{5}, we consider the case that $G=\GL_2(\FF_p)$ for $p>2$. We first give a list of all subgroups of $G$ satisfying local conditions for inertia subgroups, which are called \emph{inertia candidates}. Then we apply the work of Deligne, Fontaine and Edixhoven \cite{Edixhoven, Ribet-Stein} to translate the realizability problem into the existence of modular forms that are associated to the required Galois representation. As the level, the character and the field of definition of eigenforms vary, we expect that there exist modular forms corresponding to each inertia candidate, as follows. 
\begin{conj}\label{conjGL}
  Assume $G=\GL_2(\FF_p)$ for $p>2$. Then $(G,I,p)$ is $\QQ$-realizable for every inertia candidate. In other words, the local-global principle is valid.
\end{conj}
In \S\ref{5.3}, we prove Conjecture \ref{conjGL} for inertia candidates that correspond to the modular forms of weight 2 by constructing the associated elliptic curves defined over $\QQ$. 
\begin{thm}\label{weight2}
Assume $G=\GL_2(\FF_p)$ for $p>2$. Then $(G,I,p)$ is $\QQ$-realizable when $I$ is conjugate to
$$\left\{\left(\begin{matrix} * & 0 \\ 0 & 1\end{matrix}\right)\right\}, \left\{\left(\begin{matrix} * & * \\ 0 & 1 \end{matrix}\right)\right\}$$
or a nonsplit Cartan subgroup of $\GL_2(\FF_p)$.
\end{thm}

Finally, in Section \ref{6}, we give an example (see Example \ref{example2gp}) for which the local-global principle fails. In order to construct this example, we recall the counterexample of Grunwald-Wang \cite{Wang} which says that there is no $C_8$-extension of $\QQ$ whose local completion at 2 is the unramified $C_8$-extension of $\QQ_2$. In the abelian case studied in Section \ref{3}, we can avoid this counterexample because our realizability problem allows the decomposition subgroup to vary. In Example \ref{example2gp}, we construct a 2-group $G$ and a subgroup $I$ such that there is only one subgroup suitable for the decomposition subgroup, but it falls into the Grunwald-Wang counterexample.

\section{Preliminaries} \label{2}
\subsection{Decomposition subgroup and inertia subgroup}

Let $L/K$ be a finite Galois extension of number fields or $p$-adic local fields, and $\p$ a prime ideal of $K$. For each prime $\B$ of $L$ lying over $\p$, the decomposition subgroup and the inertia subgroup at $\B$ are defined as
\begin{eqnarray*}
 D^{L/K,\B} &=& \{\sigma \in \Gal(L/K) \mid \sigma(\B)=\B\},\\
 I^{L/K,\B} &=& \{\sigma \in D^{L/K, \B} \mid \sigma \text{ acts trivially on the residue field } \OO_L/\B\}.
\end{eqnarray*}
Since $\Gal(L/K)$ acts transitively on the set of primes of $L$ lying over $\p$, the decomposition subgroups (respectively the inertia subgroups) of primes $\B$ dividing $\p$ are conjugate in $\Gal(L/K)$. We therefore call $D^{L/K,\B}$ (respectively $I^{L/K,\B}$) the decomposition (inertia) subgroup at $\p$ without noting which $\B$ is being considered. 

\begin{definition}\label{realizability}
Suppose $G$ is a finite group and $I$ a subgroup of $G$. For a number field $K$ and a prime ideal $\p$ of $K$, the triple $(G,I,\p)$ is $K$-realizable if there exists a $G$-extension $L/K$ such that the inertia subgroup at $\p$ is $I$.

Similarly, for a $p$-adic local field $K$, the pair $(G,I)$ is $K$-realizable if there exists a $G$-extension $L/K$ such that the inertia subgroup at the unique prime of $K$ is $I$.
\end{definition}

If $K$ is a number field and $L/K$ realizes the given $(G,I,\p)$, then for every prime $\B$ of $L$ lying over $\p$, $D^{L/K,\B}$ and $I^{L/K,\B}$ can be identified with the Galois group and the inertia subgroup of the extension of local completions $L_{\B}/K_{\p}$, that is
\begin{eqnarray*}
  D^{L/K,\B} &=& \Gal(L_{\B}/K_{\p})\\
  \text{and\ \ } I^{L/K, \B} &=& I^{L_{\B}/K_{\p}, \B}.
\end{eqnarray*}

\begin{definition}
Given a finite group $G$ and a rational prime $p$, a subgroup $I$ of $G$ is called an inertia candidate for $G$ and $p$ if $G$ has a subgroup $D$ containing $I$ such that $(D,I)$ is $\QQ_p$-realizable.
\end{definition}
 Thus, if $(G,I,p)$ is $\QQ$-realizable, then $I$ is an inertia candidate for $G$ and $p$. Throughout this paper, we will discuss in each case that whether an inertia candidate can be realized, i.e. whether the local-global principle stated in \S\ref{s1} is valid.

\subsection{Structure of local field extensions} \label{2.1}

Assume $K$ is an extension of $\QQ_p$ of degree $n$. In this section, we present some known results about the structure of the absolute Galois group $G_{K}$ of $K$ in order to determine all inertia candidates for the given $G$ and $p$ later in this paper.

There are two important infinite Galois extensions of $K$ whose Galois groups are well-understood: the maximal unramified extension $K^{ur}$ and the maximal tamely ramified extension $K^{tr}$. Since unramified extensions of $K$ are in one-to-one correspondence to extensions of the residue field of $K$, the Galois group $\Gal(K^{ur}/K)$ is isomorphic to $\Gal(\overline{\FF_{p^c}}/\FF_{p^c})$, where $p^c$ is the residue degree of $K$, and hence is isomorphic to $\widehat{\ZZ}$. $K^{ur}$ is contained in $K^{tr}$, thus we have the following short exact sequence of Galois groups
$$1 \to \Gal(K^{tr}/K^{ur}) \to \Gal(K^{tr}/K) \to \Gal(K^{ur}/K)\to 1.$$
Iwasawa showed in \cite{Iwasawa} that this exact sequence splits, $\Gal(K^{tr}/K)$ is isomorphic to the profinite group generated by two elements $\sigma$ and $\tau$ with the only relation
$$\tau^\sigma= \tau^{p^c},$$ and the inertia subgroup $\Gal(K^{tr}/K^{ur})$ is the closed normal subgroup generated by $\tau$. Therefore, we have the following lemma determining when $(D,I)$ can be realized over $\QQ_p$ via a tamely ramified extension

\begin{lem} \label{tame}
Given a finite group $D$ and a normal subgroup $I$, there exists a tamely ramified extension of $\QQ_p$ with Galois group $D$ and inertia subgroup $I$ if and only if there are two elements $\sigma, \tau \in D$ such that $D$ has the following presentation
$$\langle \sigma, \tau \mid \tau^e=1, \sigma^f=\tau^r, \tau^\sigma=\tau^p \rangle$$
for some integers $e,f,r$
and $I$ is the normal subgroup generated by $\tau$.
\end{lem}

Moreover, Iwasawa also showed in \cite{Iwasawa} that the following short exact sequence splits
$$1\to \Gal(\overline{K}/K^{tr}) \to G_K \to \Gal(K^{tr}/K) \to 1.$$

On the other hand, the $p$-quotients of $G_K$ are well-understood. Let $G_K(p)$ denote the maximal pro-$p$ quotient of the absolute Galois group $G_K$. The following lemma explicitly describes the structure of $G_K(p)$. 

\begin{lem}\label{max-p}\cite[Theorem 7.5.8]{NSW} Let $\mu_p$ be the group of $p$-th roots of unity.
 \begin{enumerate}
   \item If $\mu_p \not\subseteq K$, then $G_K(p)$ is a free pro-$p$ group of rank $n+1$.
   \item If $\mu_p \subseteq K$, then $G_K(p)$ is a Demu\v skin group of rank $n+2$.
 \end{enumerate}
 Here $n=[K:\QQ_p]$.
\end{lem}

\subsection{Embedding problems} \label{2.2}

Let $K$ be a field, $M/K$ a finite Galois extension with Galois group $G$ and $\varphi: G_K \to G$ the natural surjection between Galois groups. Given a finite group extension $\widetilde{G}$ of $G$ defined by $\kappa: \widetilde{G} \twoheadrightarrow G$, we could ask the question of whether or not there is a homomorphism $\widetilde{\varphi}: G_K \to \widetilde{G}$, which extends $\varphi$ via $\kappa$ so that the following diagram commutes:
\begin{equation}\label{Embedding}
\begin{tikzcd}
 & & & G_K \arrow[dashed,swap]{dl}{\widetilde{\varphi}} \arrow{d}{\varphi} & \\
1 \arrow{r} & H \arrow{r} & \widetilde{G} \arrow{r}{\kappa} & G \arrow{r}  &1
\end{tikzcd}
\end{equation}
This is called the \emph{embedding problem} $\E (\varphi, \kappa)$ attached to $\varphi$ and $\kappa$. The homomorphism $\widetilde{\varphi}$ is called a \emph{solution of} $\E(\varphi, \kappa)$ and the corresponding field $(\overline{K})^{\ker(\widetilde{\varphi})}$ a \emph{solution field of} $\E(\varphi, \kappa)$. If $\widetilde{\varphi}$ is surjective, then $\widetilde{\varphi}$ (respectively $(\overline{K})^{\ker(\widetilde{\varphi})}$) is called a \emph{proper solution} (\emph{field}) \emph{of the embedding problem} and we have $\Gal((\overline{K})^{\ker(\widetilde{\varphi})}/K)=\widetilde{G}$.

\begin{lem}\cite[Theorem 3.8]{Embedding}\cite[Theorem 8.2]{IGT}\label{Koc}
  Let $\E(\varphi, \kappa)$ be the embedding problem described in (\ref{Embedding}). Let $U$ be a subgroup of $G$, $\widetilde{U}=\kappa^{-1}(U)$ and $L=M^U$. If $H$ is abelian and $|H|$ is coprime to $|\widetilde{G}:\widetilde{U}|$, then the solvability of $\E(\varphi, \kappa)$ is equivalent to the solvability of the following embedding problem
  \begin{equation}\label{Embedding2}
\begin{tikzcd}
 & & & G_L \arrow[dashed, swap]{dl}{} \arrow{d}{\varphi|_{G_L}} & \\
1 \arrow{r} & H \arrow{r} & \widetilde{U} \arrow{r}{\kappa|_{\widetilde{U}}} & U \arrow{r}  &1.
\end{tikzcd}
\end{equation}
\end{lem}

\begin{cor} \label{Koc-p}
  Use the notation of Lemma \ref{Koc}. If $H$ is a $p$-group and $U$ is a Sylow $p$-subgroup of $G$ containing $H$, then the solvabilities of $\E(\varphi, \kappa)$ in (\ref{Embedding}) and $\E(\varphi|_{G_L}, \kappa|_{\widetilde{U}})$ in (\ref{Embedding2}) are equivalent.
\end{cor}

\begin{proof}
  If $\E(\varphi, \kappa)$ is solvable with a solution $\widetilde{\varphi}$, then the restriction of $\widetilde{\varphi}$ to $G_L$ is a solution of $\E(\varphi|_{G_L}, \kappa|_{\widetilde{U}})$.
  
  Assume $\E(\varphi|_{G_L}, \kappa|_{\widetilde{U}})$ is solvable. Since $H$ is a finite $p$-group, it has a derived series
  $$H = H_0 \rhd H_1 \rhd H_2 \rhd \cdots \rhd H_n = 1,$$
  where $H_{i+1}=[H_i, H_i]$ for $ i=0, 1, \cdots, n-1$. The successive quotients $H_i/H_{i+1}$ are abelian. Each $H_i$ is a characteristic subgroup of $H$ and hence is normal in $\widetilde{G}$ and $\widetilde{U}$. Next we consider the two families of embedding problems $\E(\varphi_i, \kappa_i)$ and $\E(\psi_i, \lambda_i)$ described in (\ref{E_i}) and $(\ref{E_i_p})$ respectively.

\begin{equation}\label{E_i}
\begin{tikzcd}
 & & & G_K \arrow[dashed, swap]{dl}{\varphi_{i+1}} \arrow{d}{\varphi_i}  & \\
1 \arrow{r} & H_i/H_{i+1} \arrow{r} & \widetilde{G}/H_{i+1} \arrow{r}{\kappa_i} & \widetilde{G}/H_i \arrow{r}  &1
\end{tikzcd}
\end{equation}

\begin{equation}\label{E_i_p}
\begin{tikzcd}
 & & & G_L \arrow[dashed, swap]{dl}{\psi_{i+1}} \arrow{d}{\psi_i} & \\
1 \arrow{r} & H_i/H_{i+1} \arrow{r} & \widetilde{U}/H_{i+1} \arrow{r}{\lambda_i} & \widetilde{U}/H_i \arrow{r}  &1
\end{tikzcd}
\end{equation}
Here $\varphi_i$, $\psi_i$, $\kappa_i$ and $\lambda_i$ are the induced maps on these quotient groups.

The solvability of $\E(\varphi|_{G_L}, \kappa|_{\widetilde{U}})$ implies the solvability of $\E(\psi_i, \lambda_i)$ for each $i$. Because $H_i/H_{i+1}$ is abelian and $|\widetilde{G}/H_i : \widetilde{U}/H_i|=|G:U|$ is coprime to $p$, we obtain that for each $i$ the embedding problems $\E(\varphi_i, \kappa_i)$ and $\E(\psi_i, \lambda_i)$ are equivalent by Lemma \ref{Koc}. Then it follows that each $\E(\varphi_i, \kappa_i)$ is solvable and the solution $\varphi_n: G_K \to \widetilde{G}/H_n=\widetilde{G}$ gives a solution of $\E(\varphi, \kappa)$.
\end{proof}

For a pro-$p$ group $P$, the Frattini subgroup of $P$, denoted by $\Phi(P)$, is defined to be the intersection of all maximal open subgroups of $P$.

\begin{lem}\label{Frattini}
Assume $\widetilde{G}$ is a finite $p$-group. If an embedding problem $\E(\varphi, \kappa)$ described in (\ref{Embedding}) has a solution $\widetilde{\varphi}$ and the kernel $H$ of the given group extension is contained in the Frattini subgroup $\Phi(\widetilde{G})$ of $\widetilde{G}$, then $\widetilde{\varphi}$ is proper.
\end{lem}

\begin{proof}
  If $\widetilde{\varphi}$ is not surjective, then the image $\widetilde{\varphi}(G_K)$ must be contained in a maximal subgroup of $\widetilde{G}$, say $A$. But $\varphi(G_K)=\kappa(\widetilde{\varphi}(G_K)) \leq \kappa(A)=(H\cdot A)/H \leq A/H$ since $H\leq \Phi(\widetilde{G})\leq A$. It contradicts the surjectivity of $\varphi$.
\end{proof}

\section{The Case of Finite Abelian Extensions} \label{3}
In this section, we consider the $\QQ$-realizability of $(G,I,p)$ for finite abelian $G$ and prove Theorem \ref{thm_abelian}. 

\begin{reptheorem}{thm_abelian}
  When $G$ is finite abelian, $(G,I,p)$ is $\QQ$-realizable if and only if $I$ is a quotient of $\ZZ_p^\times$.
\end{reptheorem}

\begin{proof}
  We first find all inertia candidates for $G$ and $p$. By the local Kronecker-Weber Theorem, the maximal abelian extension $\QQ_p^{ab}$ of $\QQ_p$ has Galois group
  \begin{eqnarray*}
    \Gal(\QQ_p^{ab}/\QQ_p) &=& \Gal(\QQ_p(\zeta_{p^{\infty}})/\QQ_p) \times \Gal(\QQ_p^{un}/\QQ_p)\\
    &\simeq& \ZZ_p^{\times} \times \widehat{\ZZ},
  \end{eqnarray*}
  where $\QQ_p(\zeta_{p^{\infty}})$ is the totally ramified extension of $\QQ_p$ obtained by joining all $p$-power roots of unity and $\QQ_p^{un}$ is the maximal unramified extension of $\QQ_p$. Every inertia candidate is a quotient of the inertia subgroup of $\Gal(\QQ_p^{ab}/\QQ_p)$, so it is a quotient of $\ZZ_p^{\times}$.
  
  Now, we show the $\QQ$-realizability of $(G,I,p)$ for every inertia candidate $I$. Let $I$ be a quotient of $(\ZZ/p^n\ZZ)^{\times}$ for some $n$. Since $\QQ(\zeta_{p^n})/\QQ$ is a field extension totally ramified at $p$ with Galois group $(\ZZ/p^n\ZZ)^{\times}$, it has a subfield $K_1/\QQ$ whose Galois group and inertia subgroup are both $I$. Assume $$G=\ZZ/n_1\ZZ \times \ZZ/n_2\ZZ \times \cdots \times \ZZ/n_k\ZZ.$$
  We can find a sequence of primes $\{q_i\}_{i=1}^k$ such that $n_i \mid  q_i-1$, $q_i\neq p$ and $q_i\neq q_j$ when $i\neq j$. Then for each $i$, there exists a subfield of $\QQ(\zeta_{q_i})/\QQ$ with Galois group $\ZZ/n_i\ZZ$. These subfields are pairwise disjoint and unramified at $p$. So the union of these subfields, denoted by $K_2/\QQ$, has Galois group $G$ and is also unramified at $p$. Let $\iota: I \hookrightarrow G$ be the given embedding and $\pi$ a surjective homomorphism defined as
  \begin{eqnarray*}
  \pi: \Gal(K_1K_2/\QQ) =\Gal(K_1/\QQ)\times \Gal(K_2/\QQ) &\longrightarrow& G\\
  (a,b) &\mapsto& \iota(a)b.
  \end{eqnarray*}
  Then the field fixed by the kernel $K=(K_1K_2)^{\ker \pi}$ realizes $(G,I,p)$ over $\QQ$.
\end{proof}

\begin{remark}
 From the proof of Theorem \ref{thm_abelian}, we see that the local-global principle is valid when $G$ is finite abelian.
\end{remark}


\section{The Case of Extensions of Odd Degree} \label{4}

In this section, we assume $G$ is of odd order. First, the following lemma by Neukirch establishes the local-global principle in this case .
\begin{lem} \label{Neukirch}\cite{Solvable}
Suppose $K$ is a global field with a prime $\p$ and $G$ is a finite solvable group of exponent coprime to the number of roots of unity in $K$. Given a local field extension $L_{\B}/K_{\p}$ whose Galois group can be embedded into $G$, there exists a global $G$-extension $L/K$ such that $L$ has the given local completion $L_{\B}$.
\end{lem}

Since every group of odd order is solvable, Lemma \ref{Neukirch} implies that it suffices to determine all inertia candidates. Our idea to prove Theorem \ref{conclusion_odd} is to divide every Galois extension $L/\QQ_p$ into a tower of extensions $L/K/\QQ_p$, where $K$ is the maximal tamely ramified extension inside $L$. The structure of $\Gal(K/\QQ_p)$ is described in Lemma \ref{tame}, and $\Gal(L/K)$ is a $p$-group as $L/K$ is wildly ramified. Our strategy is to first study the finite $p$-extensions of $K$ and then apply techniques from embedding problems to combine them with $K/\QQ_p$ to form Galois extensions of $\QQ_p$.

\begin{lem} \label{structure}
 Let $K/\QQ_p$ be a degree $n$ extension and $\D$ a finite $p$-group with a normal subgroup $\I$. Assume $p$ is odd and $K$ does not contain the $p$-th roots of unity. Then $(\D, \I)$ is $K$-realizable if and only if $\D$ is generated by $n+1$ elements $\sigma, x_1, \cdots, x_n$ and $\I$ is the normal closure of $x_1, \cdots, x_n$ in $\D$. 
\end{lem}

\begin{proof}
  Since $K$ does not contain the $p$-th roots of unity, it follows from Lemma \ref{max-p} that the Galois group $G_K(p)$ of the maximal $p$-extension of $K$ is a free pro-$p$ group generated by $n+1$ elements. Let $\J$ be the inertia subgroup in $G_K(p)$. Then $\J$ is a normal subgroup such that the quotient $G_K(p)/\J$ is isomorphic to $\ZZ_p$, as the maximal unramified pro-$p$ extension of $K$ has Galois group $\ZZ_p$. We define a quotient map
  $$\phi: G_K(p) \to G_K(p)/\Phi(G_K(p)) \simeq (C_p)^{n+1}.$$
  The image $\phi(\J)$ is isomorphic to $(C_p)^n$. Therefore by Burnside's basis theorem we can find a generator set $\{s, a_1,\cdots, a_n\}$ of $G_K(p)$ such that $\phi(\J)$ is the normal subgroup generated by $\phi(a_1), \cdots, \phi(a_n)$. Let $N$ be the closed normal closure of $\langle a_1,\cdots, a_n \rangle$ in $G_K(p)$, i.e. 
  $$N= \langle a_1, \cdots, a_n \rangle ^{G_K(p)}=\langle s^m a_i s^{-m} \mid m\in \ZZ \text{ and } i=1, \cdots, n \rangle.$$
  $N$ is a subgroup of $\J$, because $\J$ is normal in $G_K(p)$ and contains $a_1, \cdots, a_n$. Then $G_K(p)/\J$ is a quotient of $G_K(p)/N= \langle \bar{s} \rangle \simeq \ZZ_p$, where $\bar{s}$ is the image of $s$ in $G_K(p)/N$. Note the fact that the quotient of $\ZZ_p$ by any nontrivial normal subgroup is finite. So $\J=N$.
  
  If there exists an extension $L/K$ with Galois group $\D$ and inertia subgroup $\I$, then we have a surjection $\pi: G_K(p)\to \D$ such that $\pi(\J)=\I$. Let $\sigma = \pi(s)$ and $x_i=\pi(a_i)$ for $i=1, \cdots, n$. It follows that $\D$ is generated by $\sigma, x_1, \cdots, x_n$ while $\I$ is the normal closure of $x_1, \cdots, x_n$.
  
  On the other hand, assume $\D$ is generated by $\sigma, x_1, \cdots, x_n$ and $\I$ is the normal closure of $x_1, \cdots, x_n$. Since $G_K(p)$ is free, we can construct a surjection $\pi: G_K(p) \to \D$ mapping $s$ to $\sigma$ and $a_i$ to $x_i$. Then the field fixed by $\ker(\pi)$ is a Galois extension of $K$ with Galois group $\D$ and inertia subgroup $\I$.
\end{proof}

If $(D,I)$ is $\QQ_p$-realizable, then by the structure of ramification subgroups of $D$, we know $I$ and its Sylow $p$-subgroup $I_p$ are normal in $D$. The Frattini subgroup $\Phi(I_p)$ is also normal in $D$, since $\Phi(I_p)$ is a characteristic subgroup of $I_p$. 
\begin{thm}\label{main_odd}
  Suppose that $D$ is a finite group of odd order with a subgroup $I$ and $p$ is an odd prime. $(D,I)$ is $\QQ_p$-realizable if and only if $(D/\Phi(I_p), I/\Phi(I_p))$ is $\QQ_p$-realizable. 
\end{thm}

\begin{proof}
  It is clear that the $\QQ_p$-realizability of $(D/\Phi(I_p), I/\Phi(I_p))$ follows from the $\QQ_p$-realizability of $(D,I)$.
  
  Conversely, assuming $(D/\Phi(I_p), I/\Phi(I_p))$ is $\QQ_p$-realizable with a solution $L/\QQ_p$, we can study the realizability of $(D,I)$ via the natural embedding problem:
\begin{equation}\label{Emb-D}
\begin{tikzcd}
 & & & G_{\QQ_p} \arrow[dashed,swap]{dl}{\widetilde{\varphi}} \arrow{d}{\varphi} & \\
1 \arrow{r} & \Phi(I_p) \arrow{r} & D \arrow{r}{\kappa} & D/\Phi(I_p) \arrow{r}  &1.
\end{tikzcd}
\end{equation}
  To prove $(D,I)$ is $\QQ_p$-realizable, it suffices to show the embedding problem in (\ref{Emb-D}) is 
  \begin{enumerate}
    \item[(i)] solvable,
    \item[(ii)] properly solvable, and
    \item[(iii)] there is a proper solution field having the inertia subgroup exactly $I$.
  \end{enumerate}
  In fact, we can prove that (ii) and (iii) follow from (i). First, because $\Phi(I_p)$ is a subgroup of $\Phi(D)$, Lemma \ref{Frattini} indicates (i) $\Longrightarrow$ (ii). Next, let's show (ii) implies (iii). Suppose $M/\QQ_p$ is a proper solution field of (\ref{Emb-D}) and $J$ is its inertia subgroup. Note that $\kappa(J)$ is the inertia subgroup of $L/\QQ_p$ which is $I/\Phi(I_p)$. By Burnside's basis theorem, $I_p$ is a subgroup of $J$ since $\kappa(J)$ contains the Frattini quotient $I_p/\Phi(I_p)$. So $\ker(\kappa)\subseteq J$ and $\kappa(J)=I/\Phi(I_p)$, which implies that $J=I$ and $M/\QQ_p$ realizes $(D,I)$.
  
  Now, we just need to show (i). Let $D_p$ be a Sylow $p$-subgroup of $D$ and $K$ the subfield of $L$ fixed by $D_p/\Phi(I_p)$. Then $L/K$ is an extension with Galois group $D_p/\Phi(I_p)$ and inertia subgroup $(D_p \cap I)/\Phi(I_p)=I_p/\Phi(I_p)$. By Corollary \ref{Koc-p}, the solvability of (\ref{Emb-D}) is equivalent to the solvability of the following embedding problem:
\begin{equation}\label{Emb-Dp}
\begin{tikzcd}
 & & & G_{K} \arrow[dashed,swap]{dl}{} \arrow{d}{\varphi|_{G_K}} & \\
1 \arrow{r} & \Phi(I_p) \arrow{r} & D_p \arrow[swap]{r}{\kappa|_{D_p}} & D_p/\Phi(I_p) \arrow{r}  &1.
\end{tikzcd}
\end{equation}
Recall that we assume $|D|$ and $p$ are odd, so $L$ does not contain the $p$-th roots of unity. Let $n$ denote the degree of $K/\QQ_p$. The Galois group and the inertia subgroup of $L/K$ must satisfy the statement in Lemma \ref{structure}. Namely, there are $n+1$ elements $\sigma, x_1, \cdots, x_n$ generating $\Gal(L/K)=D_p/\Phi(I_p)$ such that the inertia subgroup $I_p/\Phi(I_p)$ is generated by the conjugates of $x_i$ by the group $\langle \sigma \rangle$. Let $\sigma', x_1', \cdots, x_n'$ be elements in $D_p$ such that their images under $\kappa$ are $\sigma, x_1, \cdots, x_n$ respectively. Since $\Phi(I_p)$ is contained in $\Phi(D_p)$, Burnside's basis theorem implies $\sigma', x_1', \cdots, x_n'$ form a generator set of $D_p$. Similarly, one can check that $I_p$ is generated by
$$\{(\sigma')^{-m}x_i' (\sigma')^m \mid i=1,2,\cdots, n \text{ and } m\in \ZZ\}.$$
Again by Lemma \ref{structure}, there exists a $p$-extension $M/K$ with Galois group $D_p$ and inertia subgroup $I_p$, which means (\ref{Emb-Dp}) is solvable.
\end{proof}

Assume $p$ is odd and $K/\QQ_p$ is a tamely ramified Galois extension of odd degree $n$. By Lemma \ref{tame}, the Galois group of $K/\QQ_p$ has the presentation
$$H=\langle \tau, \sigma \mid \tau^e=1, \sigma^f=\tau^r, \tau^\sigma=\tau^p \rangle$$
and the inertia subgroup is generated by $\tau$. Let $M$ be the maximal $p$-elementary abelian extension of $K$. Since $\QQ_p(\mu_p)$, which is of even degree over $\QQ_p$, is not contained in $K$, it follows from Lemma \ref{max-p} that $\Gal(M/K)\simeq (C_p)^{n+1}$.

Define 
$$\Lambda_1(H):=\langle t,s \mid t^e=1, s^{fp}=t^{rp}, t^s=t^p\rangle,$$
which is a nonsplit extension of $H$ fitting in the following short exact sequence
\begin{eqnarray*}
  1 \longrightarrow C_p \longrightarrow \Lambda_1(H) &\overset{\varphi_1}{\longrightarrow}& H \longrightarrow 1 \\
  t & \longmapsto& \tau\\
  s & \longmapsto& \sigma.
\end{eqnarray*}
Define $\Lambda_2(H)$ to be the regular wreath product $C_p\wr H$ and $\varphi_2: \Lambda_2(H)\to H$ the natural surjection. $\Lambda_2(H)$ can be viewed as the maximal split extension of $H$ by a elementary $p$-abelian group which is generated by the orbit of one element under the $H$-action. Define $\Lambda(H)$ to be the fiber product of $\Lambda_1(H)$ and $\Lambda_2(H)$ over $H$ with respect to $\varphi_1$ and $\varphi_2$, i.e. 
$$\Lambda(H)=\{(h_1, h_2) \in \Lambda_1(H)\times \Lambda_2(H) \mid \varphi_1(h_1) = \varphi_2(h_2)\}.$$

\begin{lem} \label{M} Use the notation above.
\begin{enumerate}
  \item Let $M_1$ be the maximal tamely ramified subfield in $M$. Then $\Gal(M_1/\QQ_p)\simeq \Lambda_1(H)$ and $\varphi_1$ gives the quotient map from $\Gal(M_1/\QQ_p)$ to $\Gal(K/\QQ_p)$.
  
  \item There exists an intermediate field $M_2$ of $M/K$ such that $M_2/K$ is totally ramified, $\Gal(M_2/\QQ_p)\simeq \Lambda_2(H)$ and $\varphi_2$ gives the quotient map from $\Gal(M_2/\QQ_p)$ to $\Gal(K/\QQ_p)$.
  \item $\Gal(M/\QQ_p)\simeq \Lambda(H)$.
\end{enumerate}
\end{lem}

\begin{proof}
  $M_1$ is the degree $p$ unramified extension of $K$, because $M$ is the maximal elementary $p$-abelian extension of $K$ and contains no totally tamely ramified extension. Then $M_1/\QQ_p$ is of ramification index $e$ and inertia degree $fp$ and $\Gal(M_1/\QQ_p)$ has a quotient map to $\Gal(K/\QQ_p)=H$. So by Lemma \ref{tame}, $\Gal(M_1/\QQ_p)$ has the presentation
$$\langle t,s \mid t^e=1, s^{fp}=t^{rp}, t^s=t^p \rangle,$$
which implies (1) in the lemma.

Local class field theory tells us that there is an abelian totally ramified extension $M_2/K$ such that $\Gal(M_2/K)$ is isomorphic to $U_1(K)/U_1(K)^p$ as $H$-modules under the norm residue map, where $U_1(K)$ is the principal unit group of $K$. \cite{Krasner} proves that $U_1(K)$ has a normal basis and therefore is isomorphic to $\ZZ_p[H]$, which implies the isomorphism 
$$\Gal(M_2/K) \simeq U_1(K)/U_1(K)^p \simeq \FF_p[H]$$
of $H$-modules. The $H$-module $\FF_p[H]\simeq \Hom (\ZZ[H], \ZZ/p\ZZ)$ is coinduced, so the second cohomology group $H^2(H,\FF_p[H])$ is trivial and the following group extension with an abelian kernel
\begin{equation}\label{ext}
1\to \Gal(M_2/K) \to \Gal(M_2/\QQ_p) \to \Gal(K/\QQ_p)\to 1
\end{equation}
splits. Hence $\Gal(M_2/\QQ_p)\simeq \FF_p[H]\rtimes H = C_p\wr H=\Lambda_2(H)$ and the surjection $\Gal(M_2/\QQ_p)\to \Gal(K/\QQ_p)$ in (\ref{ext}) is $\varphi_2$.

Since the intersection of $M_1$ and $M_2$ is $K$, we have $\Gal(M_1M_2/K)=\Gal(M_1/K)\times \Gal(M_2/K)\simeq (C_p)^{n+1}$. So $M=M_1M_2$ and $\Gal(M/\QQ_p)\simeq \Lambda(H)$.
\end{proof}

We denote the maps between Galois groups as in the following diagram
\begin{center}
\begin{tikzcd}
 \Lambda(H) = \Gal(M/\QQ_p) \arrow{r}{\gamma_1} \arrow{d}{\gamma_2} \arrow{dr}{\varrho}
& \Lambda_1(H)=\Gal(M_1/\QQ_p) \arrow{d}{\varphi_1}\\
 \Lambda_2(H)=\Gal(M_2/\QQ_p)\arrow[swap]{r}{\varphi_2} &  H=\Gal(K/\QQ_p)
\end{tikzcd}
\end{center}

\begin{lem}\label{M-inertia}
The inertia subgroup, denoted by $\I(H)$, of the extension $M/\QQ_p$ is
$$\I(H)=\{(h_1, h_2) \in \Lambda(H) \mid h_1 \in \langle t \rangle\}.$$
\end{lem}

\begin{proof}
 $\I(H)$ is mapped to the inertia subgroup of $M_i/\QQ_p$ under $\gamma_i$ for $i=1,2$. The structure of $\Lambda_1(H)$ shows that the inertia subgroup of $M_1/\QQ_p$ is the subgroup generated by $t$. Since $M_2/K$ is totally ramified, the inertia subgroup of $M_2/\QQ_p$ is $\varphi_2^{-1}(\langle \tau\rangle)$. Thus, $\I(H)$ is the subgroup of $\Lambda(H)$ consisting of elements $(h_1, h_2)$ satisfying $h_1\in \langle t \rangle$.
\end{proof}

\begin{lem}\label{pre-subgp}
Assume $H=\langle \tau, \sigma \mid \tau^e=1, \sigma^f=\tau^r, \tau^\sigma=\tau^p \rangle$ and $\N\hookrightarrow G \overset{\pi}{\twoheadrightarrow} H$ is a group extension such that $N$ is a finite $p$-group and is the normal closure of one element of $G$, i.e. $N=\langle a \rangle ^G$ for some $a\in N$. Then there are elements $s,t \in G$ such that 
$$\pi(t)=\tau, \pi(s)=\sigma \text{, and } t^s=t^p.$$
\end{lem}

\begin{proof}
 Since $p\nmid e$, there exists $t\in G$ such that $\pi(t)=\tau$ and $|t|=e$. Let $T$ be the preimage $\pi^{-1}(\langle \tau \rangle)$. Because $\langle \tau \rangle$ is normal in $H$, every subgroup of $G$ conjugate to $\langle t \rangle$ is in $T$, and it follows that the number of subgroups conjugate to $\langle t \rangle$ in $G$ is $\frac{|G|}{|\N_G(\langle t \rangle)|}=\frac{|T|}{|\N_T(\langle t \rangle)|}$ by the orbit-stabilizer theorem. Also, we have
\begin{eqnarray*}
  |\N_G(\langle t \rangle)| &=& |\pi(\N_G(\langle t \rangle))| \cdot |\N_G(\langle t \rangle) \cap N|, \\
  |\N_T(\langle t \rangle)| &=& |\tau|\cdot |\N_T(\langle t \rangle) \cap N|= |\tau| \cdot |\N_G(\langle t \rangle) \cap N|.
\end{eqnarray*}
So $\frac{|G|}{|\pi(\N_G(\langle t \rangle))|}=\frac{|T|}{|\tau|}=|N|$, and it implies $\pi(\N_G(\langle t \rangle)) = H$.

Choose an element $s$ in $\pi^{-1}(\sigma) \cap \N_G(\langle t \rangle)$, which is non-empty by the argument above. The surjection $\pi$ maps $t^s$ to $\tau^\sigma=\tau^p$, so $t^s=t^p\cdot n$ for some element $n$ in $N$. However conjugation by $s$ stabilizes the subgroup $\langle t \rangle$ which has trivial intersection with $N$. So $n$ has to be 1 and we get the relation $t^s=t^p$.
\end{proof}

\begin{lem}\label{subgp}
Under the assumption in Lemma \ref{pre-subgp}, if $N$ is a elementary $p$-abelian group, then $G$ has a subgroup $F$ satisfying one of the following conditions
\begin{enumerate}
 \item[(i)] $\pi|_{F}$ is an isomorphism mapping $F$ to $H$;
 \item[(ii)] $F$ accompanied with the map $\pi|_F$ is equivalent to $\Lambda_1(H)$ as group extensions of $H$.
\end{enumerate}
\end{lem}

\begin{proof} Let $s,t$ be the elements obtained from Lemma \ref{pre-subgp} and $F$ the subgroup of $G$ generated by $s$ and $t$. Let's consider the relations that $t$ and $s$ satisfy. First, $t^e=1$ and $t^s=t^p$ is proven in Lemma \ref{pre-subgp}. The relation $\sigma^f=\tau^r$ of $H$ indicates $s^f=t^r \cdot m$ for some $m$ in $\ker \pi$. Note that $\ker \pi$ and $\langle t \rangle$ are two disjoint normal subgroups of $F$ and hence the commutator subgroup $[\ker \pi, \langle t \rangle]$ is trivial. Then we know $s^{fp}=(t^r \cdot m)^p= t^{rp}\cdot m^p = t^{rp}$, where the last equality is obtained by the assumption that $N$ is a elementary $p$-abelian group. So the following map factors through $F$.
\begin{eqnarray*}
  \langle t,s \mid t^e=1, s^{fp}=t^{rp}, t^s=t^p\rangle &\to& H\\
  t &\mapsto& \tau\\
  s &\mapsto& \sigma
\end{eqnarray*}
Thus, $F$ has to be in one of the two cases stated in the lemma.
\end{proof}

\begin{lem}\label{subgp2}
Under the assumption in Lemma \ref{subgp}, there exists a surjection $\tilde{\varrho}: \Lambda(H)\to G$ such that $\varrho=\pi\circ \tilde{\varrho}$ and the image of $\mathcal{I}(H)$ under $\tilde{\varrho}$ is $\pi^{-1}(\langle \tau \rangle)$.
\end{lem}

\begin{proof}
  If there is $F$ in Lemma \ref{subgp} (i), then $G$ is the semidirect product $N\rtimes H$. Note that $\Lambda_2(H)$, the wreath product $C_p\wr H$, is $(C_p)^n \rtimes H$ where the normal subgroup isomorphic to $(C_p)^n$ is generated by one element, say $x$, and its conjugates. Define $\theta: \Lambda_2(H) \to N\rtimes H$ that maps $x$ to $a$ and is the identity map when restricted on $H$. One can check that $\theta$ is surjective and all relations of $C_p\wr H$ are preserved by $\theta$, hence $\theta$ is an surjective homomorphism. Moreover, by the definition of $\theta$, it follows that $\varphi_2=\pi\circ \theta$. Define $\tilde{\varrho}$ to be $\theta\circ \gamma_2 : \Lambda(H) \to G$. Then it's clear that $\tilde{\varrho}$ satisfies $\varrho=\pi \circ \tilde{\varrho}$.
\begin{center}
\begin{tikzcd}
 \Lambda(H) \arrow[swap]{d}{\gamma_2} \arrow{rr}{\gamma_1} \arrow[swap]{dr}{\tilde{\varrho}} \arrow{drr}{\varrho}& & \Lambda_1(H) \arrow{d}{\varphi_1}\\
 \Lambda_2(H) \arrow[swap]{r}{\theta} \arrow[swap, bend right]{rr}{\varphi_2} & G \arrow{r}{\pi} & H
\end{tikzcd}
\end{center}
In order to prove $\tilde{\varrho}(\mathcal{I}(H))=\pi^{-1}(\langle \tau \rangle)$, we note that the image of $\mathcal{I}(H)$ in $\Lambda_2(H)$ under the map $\gamma_2$ is $\varphi_2^{-1}(\langle \tau \rangle)$. Since $\theta$ is surjective, it maps the full preimage of $\langle \tau \rangle$ in $\Lambda_2(H)$ to the full preimage in $G$, i.e. $\theta(\varphi_2^{-1}(\langle \tau \rangle))=\pi^{-1}(\langle\tau \rangle)$, which proves $\tilde{\varrho}(\mathcal{I}(H))= \theta\circ \gamma_2 (\mathcal{I}(H))=\theta(\varphi_2^{-1}(\langle \tau \rangle))=\pi^{-1}(\langle \tau \rangle)$.
  
  Otherwise, there is $F$ in Lemma \ref{subgp} (ii). Then there exists a surjection $\lambda: N\rtimes F \to G$, where the $F$-action on $N$ in the semidirect product is the same as the conjugation by $F$ in $G$. The intersection of $N$ and $F$ is isomorphic to $C_p$, and hence acts trivially on $N$ since $N$ is abelian. So $N\rtimes F$ is the fiber product of $\Lambda_1(H)$ and $N\rtimes H$ over $H$. By the proof in case (i), $\varphi_2: \Lambda_2(H)\to H$ factors through $N\rtimes H \to H$, therefore we have the following diagram, where the map $\varrho^*$ on the dashed arrow exists because of the universal property of fiber product.
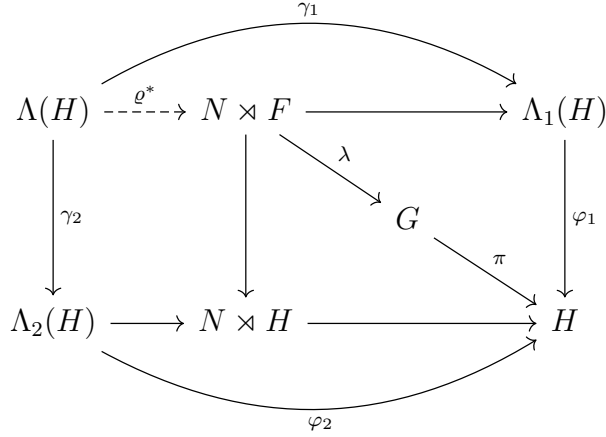
\begin{center}
\begin{tikzcd}
 \Lambda(H) \arrow[dashed]{r}{\varrho^*} \arrow{dd}{\gamma_2} \arrow[bend left]{rrr}{\gamma_1}  & N\rtimes F \arrow{rr}{} \arrow{dd}{} \arrow{dr}{\lambda}
& & \Lambda_1(H) \arrow{dd}{\varphi_1}\\
& & G \arrow{dr}{\pi} &\\
 \Lambda_2(H) \arrow{r}{} \arrow[swap, bend right]{rrr}{\varphi_2} & N\rtimes H \arrow{rr}{} &&  H
\end{tikzcd}
\end{center}
Let's consider the images of $\mathcal{I}(H)$ in the groups above. Its image in $\Lambda_1(H)$ is $\gamma_1(\mathcal{I}(H)) = \langle t \rangle$, and in $N\rtimes H$ is $N\cdot \langle \tau \rangle$. It follows that $\varrho^*(\mathcal{I}(H))$ is the product of $N$ and the cyclic subgroup of $F$ generated by $t$ (recall that we can identify $\Lambda_1(H)$ and $F$). Define the map $\tilde{\varrho}: \Lambda(H)\to G$ to be $\lambda \circ \varrho^*$. Then it's easy to check that $\varrho=\pi\circ \tilde{\varrho}$ and $\tilde{\varrho}(\mathcal{I}(H))=\lambda\circ \varrho^* (\mathcal{I}(H))=\lambda (N\cdot\langle t \rangle)= N \cdot \langle \tau \rangle = \pi^{-1}(\langle \tau \rangle)$.
\end{proof}

\begin{thm}\label{local-odd}
  Assume $D$ is a group of odd order with a subgroup $I$. $(D,I)$ is $\QQ_p$-realizable if and only if
  \begin{enumerate}
    \item (Tame Condition) $I$ is a normal subgroup of $D$ and its Sylow $p$-subgroup $I_p$ is normal. Moreover, there exist $\sigma, \tau \in D/I_p$ such that
$$D/I_p = \langle \sigma, \tau \mid \tau^e=1, \sigma^f=\tau^r, \tau^\sigma=\tau^p \rangle,$$
and $I/I_p$ is the subgroup generated by $\tau$.
    \item (Wild Condition) $I_p$ is the normal closure of one element in $D$, i.e. there is $a\in I_p$ such that
    $$I_p = \langle a \rangle^ {D}.$$
  \end{enumerate}
\end{thm}

\begin{proof}
  If $p=2$, then $I_2$ is trivial and it's obvious that $(G,I,p)$ is $\QQ$-realizable if and only if the tame condition holds.

  Suppose $p$ is odd. If $(D,I)$ is $\QQ_p$-realizable, then $(D/I_p, I/I_p)$ is $\QQ_p$-realizable via a tamely ramified extension, denoted by $K/\QQ_p$. So the tame condition holds. Also, there is an extension $L/\QQ_p$ realizing $(D/\Phi(I_p), I/\Phi(I_p))$. By Lemma \ref{M} and \ref{M-inertia}, $I_p/\Phi(I_p)$ is generated by one element and its conjugates by $D/\Phi(I_p)$. Applying Burnside's basis theorem, we obtain the wild condition.
  
Conversely, assume $G$ has a subgroup $D$ containing $I$ such satisfying (1) and (2). The tame condition and Lemma \ref{tame} promise the existence of a tamely ramified extension $K/\QQ_p$ realizing $(D/I_p, I/I_p)$. Let $H=D/I_p$ be the Galois group of $K/\QQ_p$ and again $M$ the maximal $p$-elementary abelian extension of $K$. Then $M/\QQ_p$ is an extension with Galois group $\Lambda(H)$ and inertia subgroup $\mathcal{I}(H)$. The wild condition says that the kernel of the following group extension is generated by one element and its conjugates
\begin{center}
\begin{tikzcd}
 1 \arrow{r}{} &I_p/\Phi(I_p) \arrow{r}{} & D/\Phi(I_p) \arrow{r}{\pi} & H(=D/I_p) \arrow{r}{} & 1.
\end{tikzcd}
\end{center}
It follows from Lemma \ref{subgp2} that there is a surjection $\tilde{\varrho}: \Lambda(H)\to D/\Phi(I_p)$ such that $\varrho=\pi\circ \tilde{\varrho}$ and $\tilde{\varrho}(\mathcal{I}(H))=\pi^{-1}(\langle \tau \rangle)=I/\Phi(I_p)$, which implies that there is a sub-extension of $M/\QQ_p$ realizaing $(D/\Phi(I_p), I/\Phi(I_p))$. Finally, $(D,I)$ is $\QQ_p$-realizable by Theorem \ref{main_odd}.
\end{proof}

Theorem \ref{conclusion_odd} follows immediately from Lemma \ref{Neukirch} and Theorem \ref{local-odd}. 

Furthermore, since we understand the structure of every odd-degree extension over $\QQ_p$, we are able to describe the Galois group of the maximal pro-odd extension, i.e. the maximal pro-odd quotient of $G_{\QQ_p}$. In the rest of this section, we will prove Corollary \ref{pro-odd}.

\begin{repcorollary}{pro-odd}
The Galois group of the maximal pro-odd extension of $\QQ_p$ is the pro-odd group topologically generated by three elements $\sigma, \tau, x$ with the following defining relations.
\begin{enumerate}
\item The wild inertia subgroup is the closed normal subgroup generated by $x$, which is a free pro-$p$ group.
\item The elements $\sigma, \tau$ satisfy the tame relation
$$\tau^\sigma=\tau^p.$$
\end{enumerate}
\end{repcorollary}

\begin{proof}
 If $p=2$, then the wild inertia subgroup of the maximal pro-odd extension of $\QQ_2$ is trivial, so the corollary is obvious.
 
 Suppose $p$ is odd. Let $G_1$ be the Galois group of the maximal pro-odd extension of $\QQ_p$ and $G_2$ the pro-odd group described in this corollary. $G_2$ is obviously finitely generated. Lemma \ref{local-odd} shows that every finite quotient of $G_1$ is generated by 3 elements, so the generator rank of $G_1$, which equals to the supremum of the generator ranks of all finite quotients \cite[Lemma 2.5.3]{PG}, is at most 3. It suffices to show that the sets of all finite quotients of $G_1$ and $G_2$ respectively are equal, since finitely generated profinite groups are determined by their finite quotients \cite[Proposition 15.4]{FA}. If $D$ is a finite quotient of $G_1$, i.e. there exists an odd-degree extension of $\QQ_p$ with Galois group $D$. By Theorem \ref{local-odd} and Lemma \ref{pre-subgp}, there is a generator set $\{a, s, t\}$ such that the closed normal subgroup generated by $a$ is the wild inertia subgroup and $t^s=t^p$. Define $\pi: G_2\to D$ mapping $x\mapsto a$, $\sigma\mapsto s$ and $\tau\mapsto t$. It's easy to check $\pi$ is a surjective homomorphism, so $D$ is a quotient of $G_2$. Conversely, if $D$ is a finite quotient of $G_2$, we let $I$ denote the normal subgroup of $D$ generated by the images of $x$ and $\tau$. Then the Sylow-$p$ subgroup $I_p$ is the normal closure of the image of $x$. Therefore $(D,I)$ is $\QQ_p$-realizable, since $D, I$ and $I_p$ satisfy the conditions in Theorem \ref{local-odd}. So $D$ is also a finite quotient of $G_1$.
\end{proof}

\begin{remark}
 When $p$ is odd, Corollary \ref{pro-odd} can be proven by the presentation of $G_{\QQ_p}$ given by Jannsen and Wingberg (see \cite{J-W} and the discussion in \cite{Neftin} Theorem 2.19). 
\end{remark}

\section{The Case $G=\GL_2(\FF_p)$} \label{5}

Assume $p$ is an odd prime. In this section, we consider the $\QQ$-realizability of $(\GL_2(\FF_p), I, p)$. In \S\ref{5.1}, we will find all inertia candidates for $\GL_2(\FF_p)$ and $p$. Then in \S \ref{5.2}, we will relate each inertia candidate to a type of eigenforms and conjecture that every inertia candidate is realizable (see Conjecture \ref{conjGL}). Finally, in \S \ref{5.3}, we will prove Conjecture \ref{conjGL} for the inertia candidates related to weight 2 eigenforms (see Theorem \ref{weight2}).

\subsection{Inertia candidates} \label{5.1}

Recall that a subgroup $I$ of $G=\GL_2(\FF_p)$ is an \emph{inertia candidate} if $G$ has a subgroup $D$ containing $I$ such that $(D,I,p)$ is $\QQ_p$-realizable. Before determining all inertia candidates, let us recall some facts about $\GL_2(\FF_p)$. A \emph{Borel subgroup} of $\GL_2(\FF_p)$ is a subgroup that is conjugate to the subgroup of upper-triangular matrices. There are two types of Cartan subgroups. A \emph{split Cartan subgroup} of $\GL_2(\FF_p)$ is a subgroup conjugate to the subgroup of diagonal matrices. A \emph{nonsplit Cartan subgroup} is the image of a homomorphism
$$\FF_{p^2}^{\times} \hookrightarrow \Aut_{\FF_p}(\FF_{p^2})\cong \GL_2(\FF_p),$$
where the first embedding maps $x\in\FF_{p^2}^{\times}$ to multiplication by $x$ and the isomorphism is given by some choice of $\FF_p$-basis of $\FF_{p^2}$. Therefore, a split Cartan subgroup is isomorphic to the direct product of two copies of $\FF_{p}^{\times}$ and a nonsplit Cartan subgroup is a cyclic group of order $p^2-1$. For any $\sigma\in \GL_2(\FF_p)$, $\sigma$ is contained in some (split or nonsplit) Cartan subgroup if the order of $\sigma$ is not divisible by $p$, and otherwise $\sigma$ is conjugate to $\left(\begin{smallmatrix} x & 1 \\ 0 & x \end{smallmatrix}\right)$ for some $x\in \FF_p^{\times}$.

\subsubsection{Inertia candidates of tamely ramified extensions:}\label{tame-section}

The inertia subgroup of a tamely ramified extension of $\QQ_p$ is cyclic of order coprime to $p$. Thus, if $I$ is an inertia candidate associated to a tamely ramified extension, then $I$ is cyclic and contained in some Cartan subgroup. Explicitly, $I$ is conjugate to 
$$\left\langle \left(\begin{matrix} x & 0 \\ 0 & y \end{matrix}\right)\right\rangle$$ 
for $x,y \in \FF_p^{\times}$ or $$\left\langle\left(\begin{matrix} x & y \\ \delta y & x \end{matrix}\right)\right\rangle$$
for $x,y, \delta\in \FF_p$ such that $(x,y)\neq (0,0)$ and $(\frac{\delta}{p})=-1$, where the former is in a split Cartan and the latter is in a nonsplit Cartan.

\subsubsection{Inertia candidates of wildly ramified extensions:}\label{wild-section} If $I$ is an inertia candidate associated to a wildly ramified extension,
then it has a normal Sylow-$p$ subgroup isomorphic to $C_p$ by which the quotient is cyclic, since $|\GL_2(\FF_p)|=(p-1)^2p(p+1)$. By \cite{Dickson}, when $p>2$, any subgroup of $\GL_2(\FF_p)$ with order divisible by $p$ is a subgroup of a Borel subgroup or a subgroup containing $\SL_2(\FF_p)$. Note that $\SL_2(\FF_p)$ is not solvable when $p>3$ and Sylow-3 subgroups of $\SL_2(\FF_3)$ are not normal. It follows that $I$ is contained in some Borel subgroup. The Sylow-$p$ subgroup of a Borel subgroup is conjugate to $\left(\begin{smallmatrix} 1 & 1 \\ 0 & 1\end{smallmatrix}\right)$, so every inertia candidate in the wildly ramified case is conjugate to
\begin{equation}\label{wild-form}
\left\langle \left(\begin{matrix} x & 0 \\ 0 & y \end{matrix}\right), \left(\begin{matrix} 1 & 1 \\ 0 & 1 \end{matrix}\right) \right\rangle,
\end{equation}
for some $x,y \in \FF_p^{\times}$.

\subsection{Relation with modular Galois representations} \label{5.2}

Suppose that $$f=\sum_{n=1}^{\infty}a_n q^n$$
is a normalized cuspidal eigenform of weight $k$ and character $\varepsilon$ on $\Gamma_1(N)$. Let $K_f$ be the field extension of $\QQ$ generated by the coefficients $a_n$ and the values of $\varepsilon$, and $\p$ a prime of $K_f$ lying over $p$. Then $K_f/\QQ$ is a finite extension and we can attach to $f$ a continuous mod-$p$ Galois representation
$$\rho_f: G_{\QQ} \to \GL_2(\mathcal{O}_{K_f}/\p),$$
which is unramified outside $pN$ and has the property that
$$\tr(\rho_f(\Frob_l))=a_l \text{ modulo } \p\ \  \text{   and    }\ \ \det(\rho_f(\Frob_l))=\varepsilon(l)l^{k-1}\text{ modulo } \p,$$
for all primes $l\nmid pN$. So if $K_f/\QQ$ is totally ramified at $p$, then the attached Galois representation $\rho_f$ gives a $\GL_2(\FF_p)$-extension of $\QQ$ if it is surjective. 

Let $\G^p$ and $\J^p$ denote the decomposition subgroup and the inertia subgroup of the absolute Galois group $G_{\QQ}$ at $p$, which are defined up to conjugacy. We can identify $\G^p$ with $G_{\QQ_p}$ and $\J^p$ with the inertia subgroup of $G_{\QQ_p}$. Let $\chi: G_{\QQ} \to \FF_p^{\times}$ denote the cyclotomic character defined by the action on the $p$-th roots of unity: $\sigma(\zeta_p)=\zeta_p^{\chi(\sigma)}$. Recall that $\J^p$ has a quotient $$\Gal(\QQ_p^{tr}/\QQ_p^{ur})\cong \varprojlim \FF_{p^n}^{\times}.$$ Therefore we have a natural surjection $\Psi: \J^p\to \FF_{p^2}^{\times}$ and define $\Psi'$ to be $\Psi^p$. Then $\Psi$ and $\Psi'$ are the two fundamental characters of level 2 with the property $\Psi\Psi'=\chi$. The following lemmas determine the image of $\J^p$ under the modular Galois representation when $k\geq 2$.

\begin{lem}\cite[Theorem 2.5, 2.6]{Edixhoven} \label{ResInt}
Let $f$ be an eigenform of level $N$, weight $k$ and character $\varepsilon$, and $\rho_f$ be the Galois representation attached to $f$. Assume $2\leq k \leq p+1$.
\begin{enumerate}[label=(\roman*)]
  \item {\bf Ordinary Case (Deligne)}: if $a_p\neq 0$, then $$\rho_f|_{\J^p}\sim\left(\begin{matrix} \chi^{k-1} & * \\ 0 & 1 \end{matrix}\right).$$
  \item {\bf Supersingular Case (Fontaine):} if $a_p=0$, then 
  $$\rho_f|_{\J^p} \sim \left(\begin{matrix} (\Psi)^{k-1} & 0 \\ 0 & (\Psi')^{k-1} \end{matrix}\right).$$
\end{enumerate}
Here the symbol $\sim$ means that the two sides are equal up to conjugacy.
\end{lem}

\begin{lem}\label{WeightReduce}\cite[Theorem 2.7]{Ribet-Stein} Suppose $\rho$ is modular of weight $k$, level $N$ and character $\varepsilon$, and that $p\nmid N$. Then there is some integer $a$ such that the twist $\rho \otimes \chi^{-a}$ is modular of weight $\leq p+1$, level $N$ and character $\varepsilon$.
\end{lem}

In the rest of Section \ref{5.2}, we will discuss whether there are any inertia candidates that can never be realized via modular Galois representations.

\subsubsection{Tamely Ramified Case:} Let $\alpha$ be a generator of $\FF_p^{\times}$. Recall in Section \ref{tame-section} we discussed that if $I$ is an inertia candidate associated to a $\GL_2(\FF_p)$-extension tamely ramified at $p$, then $I$ is conjugate to either $\left\langle \left(\begin{smallmatrix} \alpha^a & 0 \\ 0 & \alpha^b\end{smallmatrix}\right)\right\rangle$ for some $a, b$ or a subgroup of a nonsplit Cartan subgroup. For convenience, we say \emph{an eigenform $f$ has property $(\dagger)$} if $f$ is defined over a number field $K_f/\QQ$, which is totally ramified at $p$, and its mod-$p$ Galois representation $\rho_f: G_{\QQ}\to \GL_2(\FF_p)$ is surjective. 
\begin{prop} \label{trcase}
\begin{enumerate}
  \item Assume $$I=\left\langle\left(\begin{matrix} \alpha^{a+b} & 0 \\ 0 & \alpha^a \end{matrix}\right)\right\rangle$$
  for some integers $a$ and $1\leq b \leq p-1$. $(\GL_2(\FF_p), I, p)$ is $\QQ$-realizable if there exists an ordinary eigenform $f$ of weight $b+1$ such that $\rho_f|_{\J^p}$ is diagonal and $\rho_f\otimes \chi^a$ has property $(\dagger)$.
  \item Assume $I$ is a subgroup of a nonsplit Cartan with index $a(p+1)+b$ for integers $0\leq a\leq p$ and $1\leq b \leq p$. $(\GL_2(\FF_p), I,p)$ is $\QQ$-realizable if there exists a supersingular eigenform $f$ of weight $b+1$ such that $\rho_f\otimes \chi^a$ has property $(\dagger)$.
\end{enumerate}
\end{prop}

\begin{proof}
  \begin{enumerate}
    \item Suppose there exists such an ordinary eigenform $f$. Let $\rho=\rho_f\otimes \chi^a$. Then $\rho: G_{\QQ}\to \GL_2(\FF_p)$ is surjective and by Lemma \ref{ResInt}
    \begin{eqnarray*}
      \rho|_{\J^p} &\sim& \chi^a \otimes \left(\begin{matrix} \chi^{b} & 0 \\ 0 & 1 \end{matrix}\right)\\
      &=& \left(\begin{matrix} \chi^{a+b} & 0 \\ 0 & \chi^a \end{matrix}\right).
    \end{eqnarray*}
    Therefore $\rho(\J^p)=\left\langle \left(\begin{smallmatrix} \alpha^{a+b} & 0 \\ 0 & \alpha^a \end{smallmatrix}\right)\right\rangle$ and the extension $\overline{\QQ}^{\ker \rho}/\QQ$ realizes $(\GL_2(\FF_p), I, p)$.

    \item Also let $\rho=\rho_f \otimes \chi^a$. Then we obtain 
    \begin{eqnarray*}
      \rho|_{\J^p} &\sim& \chi^a \otimes \left(\begin{matrix} \Psi^b & 0 \\ 0 & (\Psi')^b \end{matrix}\right)\\
      &=& \left(\begin{matrix} \Psi^{a(p+1)+b} & 0 \\ 0 & (\Psi')^{a(p+1)+b} \end{matrix}\right).
    \end{eqnarray*}
    $\rho(\J^p)$ has index $a(p+1)+b$ in the nonsplit Cartan, so we showed $\overline{\QQ}^{\ker \rho}/\QQ$ realizes $(\GL_2(\FF_p),I,p)$
  \end{enumerate}
\end{proof}

\begin{remark}
  In fact, every inertia candidate associated to tamely ramified extensions satisfies the assumption in (1) or (2) in Proposition $\ref{trcase}$. If $I$ is a subgroup of a nonsplit Cartan with index $a(p+1)$, then $I$ is contained in the center of $\GL_2(\FF_p)$ and is covered by Proposition $\ref{trcase}$(1).
\end{remark}

\subsubsection{Wildly Ramified Case:} 

For inertia candidates associated to wildly ramified extensions, we will prove the following proposition.

\begin{prop}\label{wrcase}
Assume $I$ is an inertia candidate associated to wildly ramified extensions, i.e.
$$I=\left\langle \left(\begin{matrix} \alpha^{a+b} & 0 \\ 0 & \alpha^a\end{matrix} \right) , \left(\begin{matrix} 1 & 1 \\ 0 & 1 \end{matrix}\right) \right\rangle,$$ 
for $1\leq b \leq p-1$. $(\GL_2(\FF_p),I, p)$ is $\QQ$-realizable if there exists an ordinary eigenform $f$ of weight $b+1$ such that $\rho_f|_{\J^p}$ is not diagonal and $\rho_f\otimes \chi^a$ has property $(\dagger)$.
\end{prop}

Let $\beta=\alpha^{\frac{p-1}{(b, p-1)}}$. In other words, $\beta$ is a generator of the kernel of
\begin{eqnarray*}
  \varpi: \FF_p^{\times} &\to& \FF_p^{\times} \\
  x &\mapsto& x^{b}.
\end{eqnarray*}

\begin{lem} \label{lemma}
If $f$ is an ordinary eigenform of weight $b+1$ such that $\rho_f$ is not diagonal, then \begin{enumerate}
  \item there exists $g \in \J^p$ such that $\chi(g)=\beta$ and $\rho_f(g)=\left(\begin{matrix} 1 & 1 \\ 0 & 1 \end{matrix}\right)$; and
  \item there exists $h\in \J^p$ such that $\chi(h)=\alpha$ and $\rho_f(h)=\left(\begin{matrix} \alpha^{b} & 0 \\ 0 & 1 \end{matrix}\right)$.
\end{enumerate}

\end{lem}

\begin{proof}
By Lemma \ref{ResInt}
$$\rho_f|_{\J^p} = \left(\begin{matrix} \chi^{b} & * \\ 0 & 1 \end{matrix}\right) \text{  , which implies   } \rho_f(\J^p)= \left\langle \left(\begin{matrix} \alpha^{b} & 0 \\ 0 & 1 \end{matrix} \right), \left(\begin{matrix} 1 & 1 \\ 0 & 1 \end{matrix}\right) \right\rangle.$$

 \begin{enumerate}
   \item[(a)] First, there is an element $g_1\in \J^p$ such that $\chi(g_1)=\beta$. Then $\rho_f(g_1)=\left(\begin{matrix}1 & m\\ 0 & 1 \end{matrix}\right)$ for some $m$. If $m\neq 0$, then we can find an integer $A$ such that $Am\equiv 1\ (\Mod p)$ and $A\equiv 1 \ (\Mod p-1)$. Let $g=g_1^A$. So  \begin{eqnarray*}
    \rho_f(g) &=& \rho_f(g_1)^A = \left(\begin{matrix}1 & m \\ 0 & 1 \end{matrix}\right)^A = \left(\begin{matrix}1 & 1\\ 0 & 1\end{matrix}\right)\\
    \chi(g) &=& \chi(g_1)^A = \beta^A = \beta.
   \end{eqnarray*}
   
   Otherwise, if $m=0$, then $\rho_f(g_1)=\left(\begin{matrix}1 & 0\\ 0 & 1\end{matrix}\right)$. We can find an element $g_2\in \J^p$ such that $\rho_f(g_2)=\left(\begin{matrix} 1 & 1\\ 0 & 1\end{matrix}\right)$, and assume $\chi(g_2)=\beta^B$ since $\beta$ generates $\ker(\varpi)$ and $\chi(g_2)\in \ker(\varpi)$. Let $g=g_2 g_1^{p-B}$. So \begin{eqnarray*}
    \rho_f(g) &=& \rho_f(g_2)\rho_f(g_1)^{p-B} = \left(\begin{matrix} 1 & 1\\ 0 & 1\end{matrix}\right) \left(\begin{matrix}1 & 0 \\ 0 & 1 \end{matrix}\right)^{p-B} = \left(\begin{matrix}1 & 1\\ 0 & 1\end{matrix}\right)\\
    \chi(g) &=& \chi(g_2)\chi(g_1)^{p-B} = \beta^B \beta^{p-B} = \beta^p =\beta.
   \end{eqnarray*}
   
   \item[(b)] There exists $h_1\in \J^p$ such that $\chi(h_1)=\alpha$. Assume $\rho_f(h_1)=\left(\begin{matrix}\alpha^{b} & n\\ 0 & 1 \end{matrix}\right)$ for some integer $n$ and assume $h=g^{np-n} h_1$ where $g$ is the element constructed in (a). Then
   \begin{eqnarray*}
     \rho_f(h) &=& \rho_f(g)^{np-n} \rho_f(h_1) = \left(\begin{matrix}1 & np-n \\ 0 & 1\end{matrix}\right) \left(\begin{matrix}\alpha^{b} & n \\ 0 & 1 \end{matrix}\right) =\left(\begin{matrix} \alpha^{b} & 0 \\ 0 & 1\end{matrix}\right)\\
     \chi(h) &=& \chi(g)^{np-n}\chi(h_1)=\beta^{n(p-1)}\alpha = \alpha
   \end{eqnarray*}
 \end{enumerate}
\end{proof}

\begin{proof}[Proof of Proposition \ref{wrcase}]

Suppose there exists such an ordinary eigenform $f$. Let $g$ and $h$ be the elements obtained in Lemma \ref{lemma}. Denote $\rho=\rho_f\otimes \chi^a$. Since
\begin{eqnarray*}
  \rho(h) &=& \rho_f\otimes \chi^a (h) = \left(\begin{matrix} \alpha^{a+b} & 0 \\ 0 & \alpha^{a}\end{matrix}\right)\\
  \rho(g) &=& \rho_f\otimes \chi^a (g)\\
  &=& \left(\begin{matrix}\beta^a & 0 \\ 0 & \beta^a \end{matrix}\right) \left(\begin{matrix}1 & 1 \\ 0 & 1 \end{matrix}\right) \\
  &=& \left(\begin{matrix}\alpha^{a+b} & 0 \\ 0 & \alpha^{a} \end{matrix}\right)^{\frac{p-1}{(b,p-1)}} \left(\begin{matrix}1 & 1 \\ 0 & 1 \end{matrix}\right),
\end{eqnarray*}
$I$ is generated by $\rho(g)$ and $\rho(h)$. It suffices to show $\rho(\J^p) \subseteq I$. For every $\sigma\in \J^p$, there exist positive integers $i$ and $j$ such that $\sigma g^i h^j$ is in the kernel of $\rho_f$, and hence $\chi(\sigma g^i h^j)$ is in the kernel of $\varpi$. Let $\chi(\sigma g^i h^j)=\beta^C$. Then,
\begin{eqnarray*}
  \rho(\sigma g^i h^j) &=& \rho_f\otimes \chi^a (\sigma g^i h^j) = \left(\begin{matrix} \beta^{aC} & 0 \\ 0 & \beta^{aC} \end{matrix}\right) \\
  &=& \left(\begin{matrix}\alpha^{a+b} & 0 \\ 0 & \alpha^{a} \end{matrix}\right)^{\frac{(p-1)C}{(b,p-1)}} \in I.
\end{eqnarray*}
So $\rho(\sigma)\in I$ and we finish the proof.
\end{proof}

Proposition \ref{trcase} and \ref{wrcase} show that when $I$ is an inertia candidate, $(\GL_2(\FF_p), I, p)$ can be realized over $\QQ$ if the corresponding eigenforms exsit. Note that for an inertia candidate, we are allowed to vary the level $N$, the character $\varepsilon$ and the field of definition $K_f$ to search for the suitable eigenform $f$. It suggests us to expect the existence of eigenforms corresponding to each inertia candidate. 

\begin{repconjecture}{conjGL}
  Assume $G=\GL_2(\FF_p)$ for $p>2$. Then $(G,I,p)$ is $\QQ$-realizable for every inertia candidate. In other words, the local-global principle is valid.
\end{repconjecture}

\subsection{Realization via Elliptic Curves}\label{5.3}

Let $E$ be an elliptic curve defined over $\QQ$. For any prime $l$, if $l$ divides the discriminant $\Delta_E$ of $E$, then $E$ has bad reduction at $l$; otherwise, $E$ has good reduction. For each prime $l$, we let $\widetilde{E}_l$ denote the reduced curve over $\FF_l$ and define 
$$b_l:= l+1-\#\widetilde{E}_l(\FF_l).$$
Then the $L$-series associated to $E/\QQ$ is defined by the Euler product
$$L_E(s)=\prod_{l \mid \Delta_E}(1-b_l l^{-s})^{-1} \prod_{l \nmid \Delta_E} (1- b_l l^{-s}+l^{1-2s})^{-1}=\sum_{n=1}^{\infty} \frac{a_n}{n^s}.$$
Let $f_E(\tau)=\sum a_n e^{2\pi i n\tau}=\sum a_n q^n$ be the inverse Mellin transform of $L_E$. By the modularity theorem, $f_E$ is an eigenform of weight 2 for the congruence subgroup $\Gamma_0(N)$ where $N$ is the conductor of $E$. If the prime $p$ does not divide $\Delta_E$, then $b_p=a_p$ and the 
mod-$p$ Galois representation attached to $f_E$ is defined by the action of $G_{\QQ}$ on the $p$-torsion subgroup $E[p]$ of $E$, i.e.
$$\rho_{E,p}: G_{\QQ} \to \Aut(E[p])=\GL_2(\FF_p).$$

In this section, we prove the following theorem, which implies that every inertia candidate corresponding to weight 2 eigenforms can be realized.

\begin{thm}\label{weight2-thm} There exists an elliptic curve $E$ defined over $\QQ$ whose associated Galois representation $\rho_{E,p}$ is surjective in each of the following cases.
  \item[(a)] $E$ has supersingular good reduction at $p$.
  
  \item[(b)] $E$ has ordinary good reduction at $p$ and $\rho_{E,p}|_{\J^p}$ is diagonalizable.
  
  \item[(c)] $E$ has ordinary good reduction at $p$ and $\rho_{E,p}|_{\J^p}$ is not diagonalizable.
  
\end{thm}

To prove this theorem, we first need an effective criterion for determining whether the Galois representation $\rho_{E,p}: G_{\QQ}\to \GL_2(\FF_p)$ is surjective. According to Zywina \cite{Zywina}, if $\rho_{E,p}$ is not surjective, then the denominator of $j_E$ must be of a special form as in the following lemma.

\begin{lem} \cite[Theorem 1.5]{Zywina} \label{zywina} Let $p_1^{e_1}\cdots p_s^{e_s}$ be the factorization of the denominator of $j_E$, where the $p_i$ are distinct primes with $e_i >0$. If $\rho_{E,p}$ is not surjective for a prime $p>13$ with $$(p,j_E)\not\in S_0:=\{ (17, -17^2\cdot101^2/2), (17, -17\cdot 373^3/2^{17}), (37, -7\cdot 11^3), (37, -7\cdot 137^3 \cdot 2083^3)\},$$ then each $p_i$ is congruent to $\pm 1$ modulo $p$ and each $e_i$ is divisible by $p$.
\end{lem}

\begin{proof}[Proof of Theorem \ref{weight2-thm}(a)]
First, assume $p > 13$. It is well known that for each prime $p$, there is always a supersingular elliptic curve $\widetilde{E}/\FF_p$. Since $\text{char}(\FF_p)\neq 2, 3$, we can assume
$$\widetilde{E}/\FF_p \text{  :  } y^2 = x^3 + \bar{A}x +\bar{B}$$
for some $\bar{A}, \bar{B}\in \FF_p$. Let $A$ and $B$ be fixed integers whose images in the residue field $\FF_p$ are $\bar{A}$ and $\bar{B}$ respectively. Then any elliptic curve over $\QQ$ in the form
$$E/\QQ \text{  :  } y^2=x^3 + (A+ap)x + (B+bp)$$
is supersingular at $p$ for any integers $a, b$, and the $j$-invariant is
$$j_E=2^8\cdot 3^3\cdot \frac{(A+ap)^3}{4(A+ap)^3+27(B+bp)^2}.$$
Let $q>3$ be a prime that is not congruent to $\pm 1$ mod $p$. By the Chinese Remainder theorem, there exists an integer $a$ such that $-\frac{1}{3}(A+ap)$ is an integer congruent to 1 mod $q$, and then we have
$$-\frac{4}{27}(A+ap)^3 \equiv 4\  (\Mod q).$$
Also, we can find sufficiently large $b$ such that 
$$B+bp \equiv 2\ (\Mod q) \text{   and   } (p, j_E) \not\in S_0.$$

Under this construction of $E$, we have $4(A+ap)^3+27(B+bp)^2\equiv 0\ (\Mod q)$ and $2^8\cdot 3^3 \cdot (A+ap)^3$ is not divisible by $q$. By Lemma \ref{zywina}, $\rho_{E,p}$ is surjective. Then we proved the existence of $E/\QQ$ in Theorem \ref{weight2-thm}(a) for $p>13$.

For the small primes $p=3,5,7,11$ and $13$, we searched for elliptic curves in the LMFDB database. The following list gives elliptic curves $E_p/\QQ$ that have supersingular reduction at $p$ and a surjective representation $\rho_{E_p,p}$ for each of these small primes. 

\begin{center}
  \begin{tabular}{  c | l  }
    \hline
    $p$ & elliptic curve $E_p/\QQ$ \\ \hline \hline
    3 & 17.a1 : $y^2+xy+y =x^3-x^2-91x-310$\\ \hline
    5 & 14.a1 : $y^2+xy+y=x^3-2731x-55146$\\ \hline
    7 & 15.a1 : $y^2+xy+y=x^3+x^2-2160x-39540$\\ \hline
    11 & 14.a1 : $y^2+xy+y=x^3-2731x-55146$\\ \hline
    13 & 56.b1 : $y^2=x^3-x^2-40x-84$\\ \hline
    \hline
  \end{tabular}
\end{center}
\end{proof}

For an elliptic curve $E/\QQ$ with ordinary reduction at $p$, Gross in \cite{Gross} discussed when the restriction of $\rho_{E,p}$ to a decomposition subgroup $\G^p$ at $p$ is diagonalizable.

\begin{lem} \label{Gross}\cite[\S 17]{Gross} Assume $E/\QQ$ has ordinary reduction at $p$. Let $j_E$ be the modular invariant of $E$ in $\ZZ_p$, and let $j^{\uparrow}$ be the canonical lifting of the reduction of $j_E$ modulo $p$. Then the restriction of $\rho_{E,p}$ to $\G^p$ is diagonalizable if and only if $j_E \equiv j^{\uparrow}\ (\Mod p^2)$ for odd $p$, and $j_E \equiv j^{\uparrow}\ (\Mod 8)$ for $p=2$.
\end{lem}

Recall that when $f$ is an ordinary weight 2 eigenform with level $N$ not divisible by $p$, we know
$$\rho_{f}(\J^p) = \left\{\left(\begin{matrix} * & 0 \\ 0 & 1 \end{matrix}\right)\right\} \text{\ \  or\ \ } \left\{\left(\begin{matrix} * & * \\ 0 & 1 \end{matrix}\right)\right\}.$$
If $\rho_{E,p}|_{\G^p}$ is diagonalizable, then $\rho_{E,p}(\J^p)$ equals the first subgroup. Suppose that $\rho_{E,p}|_{\G^p}$ is not diagonalizable but $\rho_{E,p}|_{\J^p}$ is diagonalizable. Then we obtain a contradiction because $\rho_{E,p}(\J^p)$ is normal in $\rho_{E,p}(\G^p)$ but the subgroup $\left\{\left(\begin{smallmatrix} * & 0 \\ 0 & 1 \end{smallmatrix}\right)\right\}$ is not stabilized by any upper triangular matrix with nonzero top-right entry. Thus we can replace $\G^p$ by $\J^p$ in Lemma \ref{Gross}.

\begin{proof}[Proof of Theorem \ref{weight2-thm} (b,c)]
  
 (b) We first assume $p>13$, and assume $y^2=x^3+Ax+B$ has ordinary reduction at $p$ with $p\nmid  A$. Consider the elliptic curve
    $$E/\QQ \text{  :  } y^2 = x^3 + (A+ap)x + (B+bp).$$
Let $q>3$ be a fixed prime not congruent to $\pm 1$ mod $p$, and $a$ a fixed integer such that $-\frac{1}{3}(A+ap) \equiv 1\ (\Mod q)$. Now only $b$ varies, so the $j$-invariant of $E$ is a function of $b$,
$$j_E(b)= 2^8\cdot 3^3 \cdot \frac{(A+ap)^3}{4(A+ap)^3+27 (B+bp)^2}.$$
Since we assume $p\nmid A$, $j_E(b)$ is not divisible by $p$. Let $j^\uparrow(b)$ be the canonical lifting of the reduction of $j_E(b)$ modulo $p$. Since $j^\uparrow(b)\equiv j_E(b)\ (\Mod p)$, we have
\begin{eqnarray}  
  & & j^\uparrow(b)\cdot[4(A+ap)^3 + 27(B+bp)^2] \equiv 2^8\cdot 3^3 \cdot(A+ap)^3 \  (\Mod p) \nonumber\\
  &\Longrightarrow& 27j^\uparrow(b)B^2 \equiv [2^8\cdot 3^3-4j^\uparrow(b)] (A+ap)^3 \ (\Mod p)  \label{1}
\end{eqnarray}
Thus, $j^\uparrow(b)$ does not depend on $b$, and we denote it by $j^\uparrow$.

Our goal is to adjust $b$ such that $j_E(b)\equiv j^\uparrow\ (\Mod p^2)$ and the denominator of $j_E(b)$ is divisible by $q$. By (\ref{1}), we can assume $(2^8\cdot 3^3 -4 j^\uparrow)(A+ap)^3-27j^\uparrow B^2 =np$ for some integer $n$. Then by the Chinese Remainder Theorem, there exists an integer $b$ such that 
\begin{equation}\label{b}
 \begin{cases}
  b \equiv n(54 j^\uparrow)^{-1}\ (\Mod p) \\
  B+bp \equiv 2\ (\Mod q)
\end{cases}
\end{equation}
and we can make $b$ sufficiently large such that $(p,j_E(b))\not\in S_0$. Then, because $2^8\cdot 3^3 \cdot(A+ap^3)\not\equiv 0\ (\Mod q)$ and $4(A+ap)^3+27(B+bp)^2 \equiv 0\ (\Mod q)$, Lemma \ref{zywina} implies $\rho_{E,p}$ is surjective. It's easy to check $j_E(b)\equiv j^\uparrow (\Mod p^2)$ under our construction. By Lemma \ref{Gross}, $\rho_{E,p}|_{\J^p}$ is diagonalizable, which proves the existence in Theorem \ref{weight2-thm}(b) for $p>13$. 
  
  For each prime $p\leq13$, we again search in the LMFDB database for an elliptic curve $E_p$ such that it has ordinary reduction at $p$, $\rho_{E_p,p}$ is surjective, and $j_{E_p}$ is congruent to its canonical lift $j^\uparrow$ mod $p^2$. When computing $j^\uparrow$ we use the formula given in \cite{Finotti}. The following list gives information of $E_p$ for each $p$.
\begin{center}
  \begin{tabular}{  c | l | c | c }
    \hline
    $p$ & elliptic curve $E_p/\QQ$  & $j_{E_p}$ & $j^\uparrow$ $\Mod p^2$\\ \hline \hline
    3 & 89.a1 : $y^2+xy+y=x^3+x^2-x$  & $- 7^6 \cdot 89^{-1}$ & 1\\ \hline
    5 & 17.a2 : $y^2+xy+y=x^3-x^2-6x-4$ & $3^3\cdot 7^3 \cdot 13^3\cdot 17^{-2}$ & 3\\ \hline
    7 & 17.a1 : $y^2+xy+y=x^3-x^2-91x-310$ & $3^3\cdot 17^{-1} \cdot 1451^3$ & 19\\ \hline
    11 & 54.a3 : $y^2+xy=x^3-x^2+12x+8$ & $2^{-3}\cdot 3^3 \cdot 7^3$ & 114 \\ \hline
    13 & 14.a3 : $y^2+xy+y = x^3 -36x-70$ & $2^{-3}\cdot 5^3 \cdot 7^{-6}\cdot 11^3 \cdot 31^3$ & 38\\ \hline
    \hline
  \end{tabular}
\end{center}
  
(c) For $p>13$, define $A, B, a, q, n$ and $E/\QQ$ as in the proof of (b). Let $j_0$ be an integer satisfying $j_0\equiv j^\uparrow\ (\Mod p)$ and $j_0\not\equiv j^\uparrow \ (\Mod p^2)$. Instead of (\ref{b}), we choose $b$ to be a sufficiently large integer such that $(p,j_E)\not\in S_0$ and
 \begin{equation*}
 \begin{cases}
  b \equiv n(54 j_0)^{-1}\ (\Mod p) \\
  B+bp \equiv 2\ (\Mod q)
\end{cases}
\end{equation*}
Then applying a similar computation as in (b), we know $j_E\equiv j_0 \not \equiv j^\uparrow \ (\Mod p^2)$. So $\rho_{E,p}|_{\J^p}$ is not diagonalizable. 

For $p\leq 13$, we also give a list of elliptic curves $E_p$, ordinary at $p$, such that the $j$-invariant of $E_p$ is not congruent to its canonical lift $j^\uparrow$ mod $p$ and the representation $\rho_{E_p,p}$ is surjective.
\begin{center}
  \begin{tabular}{  c | l | c | c | c }
    \hline
    $p$ & elliptic curve $E_p/\QQ$  & $j_{E_p}$ &  $j_{E_p} \Mod p^2$ & $j^\uparrow$ $\Mod p^2$\\ \hline \hline
    3 & 11.a2 : $y^2+y=x^3-x^2-10-20$  & $- 2^{12} \cdot 11^{-5} \cdot 31^3$ & 7 & 1\\ \hline
    5 & 19.a2 : $y^2+y=x^3+x^2-9x-15$ & $- 2^{18}\cdot 7^3 \cdot 19^{-3}$ & 12 & 2 \\ \hline
    7 & 11.a2 : $y^2+y=x^3-x^2-10-20$  & $- 2^{12} \cdot 11^{-5} \cdot 31^3$ & 18 & 46 \\ \hline
    11 & 19.a2 : $y^2+y=x^3+x^2-9x-15$ & $- 2^{18}\cdot 7^3 \cdot 19^{-3}$ & 65 & 43\\ \hline
    13 &11.a2 : $y^2+y=x^3-x^2-10-20$  & $- 2^{12} \cdot 11^{-5} \cdot 31^3$& 88 & 10 \\ \hline
    \hline
  \end{tabular}
\end{center}
\end{proof}

Therefore, we proved Conjecture \ref{conjGL} in the following cases.

\begin{reptheorem}{weight2}
Assume $G=\GL_2(\FF_p)$ for $p>2$. Then $(G,I,p)$ is $\QQ$-realizable when $I$ is conjugate to
$$\left\{\left(\begin{matrix} * & 0 \\ 0 & 1\end{matrix}\right)\right\}, \left\{\left(\begin{matrix} * & * \\ 0 & 1 \end{matrix}\right)\right\}$$
or a nonsplit Cartan subgroup of $\GL_2(\FF_p)$.
\end{reptheorem}

\begin{proof}
 Because there exists an elliptic curve $E/\QQ$ described in Theorem \ref{weight2-thm}(a), there is an eigenform of weight 2 that has supersingular good reduction at $p$ and the associated Galois representation $\rho_{E,p}$ is surjective. Then by Lemma \ref{ResInt}, $\rho_{E,p}$ maps $\J^p$ to a nonsplit Cartan subgroup, so we obtain the $\QQ$-realizability of $(G,I,p)$ when $I$ is a nonsplit Cartan subgroup. Similarly, Theorem \ref{weight2-thm}(b) and (c) prove the $\QQ$-realizability when $I$ is conjugate to
 $$\left\{\left(\begin{matrix} * & 0 \\ 0 & 1 \end{matrix}\right)\right\} \text{  and  } \left\{\left(\begin{matrix} * & * \\ 0 & 1\end{matrix}\right)\right\}\text{  respectively}.$$
\end{proof}

\section{An Example Arising from the Grunwald-Wang Counterexample} \label{6}

The local-global principle is proven or conjectured to be valid for the cases investigated in the previous sections. In particular, unlike in the Grunwald-Wang problem, the local-global principle works for the abelian case, which is because we can vary the decomposition subgroup to avoid the Grunwald-Wang counterexample. It would be interesting to ask whether the local-global principle is always valid. In this section, we will show that the answer is negative by constructing a triple $(G,I,2)$ in Example \ref{example2gp}, where the given $G$ is the only choice of the decomposition subgroup, but it contradicts the Grunwald-Wang counterexample.

Let $\QQ_2(2)$ denote the maximal pro-2 extension of $\QQ_2$, and $G(2)$ the Galois group of $\QQ_2(2)/\QQ_2$. 

\begin{lem}
There are elements $x,y,z \in G(2)$ such that 
$$G(2)=\langle x,y,z \mid x^2y^4(y,z)=1 \rangle,$$
and the inertia subgroup of $\QQ_2(2)/\QQ_2$ is the normal subgroup generated by $xy^2$ and $z$. 
\end{lem}

\begin{proof}
Since $\mu_2 \subseteq\QQ_2$, $G(2)$ is a Demu\v skin group of rank 3 which, by \cite{Serre}, is a pro-2 group generated by three elements with the defining relation $x^2y^4(y,z)=1$. Its abelianization 
$$G(2)^{ab} = G(2) / [G(2), G(2)]$$
is the abelian pro-2 group generated by $\bar{x}, \bar{y}$ and $\bar{z}$ with the relation $\bar{x}^2\bar{y}^4=1$, where $\bar{x}, \bar{y}, \bar{z}$ are images of $x,y,z$ in $G(2)^{ab}$. By local class field theory, the maximal abelian pro-2 extension of $\QQ_2$ has Galois group isomorphic to $\ZZ/2\ZZ \times (\ZZ_2)^2$ and inertia subgroup isomorphic to $\ZZ/2\ZZ \times \ZZ_2$ which shows that $\bar{x}\bar{y}^2$, the only element of order 2, is in the inertia subgroup. Also note that the inertia subgroup of $\QQ_2(2)/\QQ_2$ contains the commutator subgroup $[G(2), G(2)]$ as the maximal unramified extension of $\QQ_2$ is abelian. It follows that $xy^2$ is an element of the inertia of the extension $\QQ_2(2)/\QQ_2$.

Let $\QQ_2(\zeta_{2^{\infty}})$ be the totally ramified extension of $\QQ_2$ obtained by joining all 2-power roots of unity. Then there exists a canonical isomorphism 
\begin{eqnarray*}
  \ZZ_2^{\times} & \to & \Gal(\QQ_2(\zeta_{2^{\infty}})/\QQ_2) \\
  a & \mapsto & \sigma_a,
\end{eqnarray*}
where $\sigma_a(\zeta)=\zeta^a$ for every root of unity $\zeta$. Since $\QQ_2(\zeta_{2^{\infty}})\subset \QQ_2(2)$, there is a continuous surjective homomorphism 
$$\chi: G(2) \to \ZZ_2^{\times}.$$
\cite{Labute} shows $\chi(x)=-1$, $\chi(y)=1$ and $\chi(z)=(-3)^{-1}$, whence there exists $d$ such that $y^dz$ is in the inertia subgroup of $\QQ_2(2)/\QQ_2$. Since $(y,y^dz)=(y,z)$, we replace $z$ by $y^dz$ and obtain that the inertia subgroup of $\QQ_2(2)/\QQ_2$ is the normal subgroup generated by $xy^2$ and $z$.
\end{proof}

\begin{example} \label{example2gp}
Let $I$ be an elementary 2-abelian group of rank 4, i.e. $I\simeq (C_2)^4$, generated by $a_1, a_2, a_3, a_4$, and $H$ a cyclic group of order 8 generated by $h$. Let $\omega$ be an $H$-action on $I$ defined by $\omega(h)$ stabilizing $a_1$ and sending $a_2 \mapsto a_3$, $a_3 \mapsto a_4$, $a_4 \mapsto a_2 a_3 a_4$. There exists a unique split group extension 
$$1\to I \to G \to H \to 1,$$
where the action of $H$ on $I$ is via $\omega$. Then the local-global principle fails for $(G,I,p)$.
\end{example}

\begin{proof}
Define a map $\pi$ by
\begin{eqnarray*}
  \pi: G(2) & \to & G \\
   x & \mapsto & a_2 h^{-2}\\
   y & \mapsto & h\\
   z & \mapsto & a_1 a_2 a_3,
\end{eqnarray*}
which is a surjective homomorphism since the only relator $x^2y^4(y,z)$ of $G(2)$ maps to 1 and $\pi(x), \pi(y), \pi(z)$ generate $G$. The normal subgroup $I$ is generated by $\pi(xy^2)=a_2, \pi(z)=a_1a_2a_3$ and their conjugates in $G$, so $\pi$ maps the inertia subgroup of $G(2)$ to $I$, and hence we obtain the $\QQ_2$-realizability of $(G,I,2)$. 

If $(G,I,2)$ is $\QQ$-realizable, then there exists $D\leq G$ such that $(D,I)$ is $\QQ_2$-realizable. So we see that $D$ can be generated by 3 elements as it is a quotient of $G(2)$. Since $G/I$ is cyclic of order 8, we have four choices of $D$, which are the preimages of the four subgroups of $G/I$. Using Magma, we check the generator ranks of these subgroups and see that $G$ is the only one of generator rank less than or equal to 3. However, there is no number field extension realizing $(G,I,2)$ over $\QQ$, because if there is such an extension $L/\QQ$, then it has a subextension $K/\QQ$, where $K=L^I$, such that $\Gal(K/\QQ)=C_8$ and the completion $K_2/\QQ_2$ is unramified of order 8, which contradicts the Grunwald-Wang counterexample. Therefore, $(G,I,2)$ is not $\QQ$-realizable.
\end{proof}

\subsection*{Acknowledgements}
I would like to thank my advisor Nigel Boston for all his help and guidance throughout this research. This work was done with the support of the
National Science Foundation grant DMS-1301690, and the
Alfred P. Sloan Foundation.

\bibliographystyle{amsalpha}
\bibliography{IGP} 

\end{document}